\documentclass{amsart}[12pt]
\usepackage{amsmath, amsthm, amscd, amsfonts,}
\usepackage{graphicx,float}
\usepackage{lineno,hyperref}
\usepackage{amssymb}

\DeclareMathOperator{\dom}{dom}

\DeclareMathOperator{\Ult}{Ult}

\def\MPB{{\mathbb{P}}}
\def\MQB{{\mathbb{Q}}}
\def\MRB{{\mathbb{R}}}

\newtheorem{theorem}{Theorem}[section]
\newtheorem{lemma}[theorem]{Lemma}

\newtheorem{corollary}[theorem]{Corollary}

\newtheorem{definition}[theorem]{Definition}

\newtheorem{claim}[theorem]{Claim}
\numberwithin{equation}{section}

\def\MPB{{\mathbb{P}}}
\def\MQB{{\mathbb{Q}}}
\def\MRB{{\mathbb{R}}}

\def\rmark{\mbox{$\rm\bf\rule{0.06em}{1.45ex}\kern-0.05em R$}}
\def\pmark{\mbox{$\rm\bf\rule{0.06em}{1.45ex}\kern-0.05em P$}}
\def\nmark{\mbox{$\rm\bf\rule{0.06em}{1.45ex}\kern-0.05em N$}}
\def\vdash{\mbox{$\rm\| \kern-0.13em -$}}

\usepackage{pb-diagram}

\def\rmark{\mbox{$\rm\bf\rule{0.06em}{1.45ex}\kern-0.05em R$}}
\def\pmark{\mbox{$\rm\bf\rule{0.06em}{1.45ex}\kern-0.05em P$}}
\def\nmark{\mbox{$\rm\bf\rule{0.06em}{1.45ex}\kern-0.05em N$}}
\def\vdash{\mbox{$\rm\| \kern-0.13em -$}}

\begin{document}

\nocite{*}

\title[The tree property at double successors of singular cardinals]{The tree property at  double successors of singular cardinals of uncountable cofinality}

\author[M. golshani and R. Mohammadpour]{Mohammad Golshani and Rahman Mohammadpour}


 \maketitle

\begin{abstract}
Assuming the existence of a strong cardinal $\kappa$ and a measurable cardinal above it, we force  a generic extension in which $\kappa$ is a singular strong limit cardinal of any given cofinality, and such that the tree property holds at $\kappa^{++}$.
\end{abstract}

\maketitle
\section{Introduction}
Infinite trees are of fundamental importance in modern set theory. In this paper, we are interested in $\kappa$-Aronszajn trees.
 Recall that a $\kappa$-tree $T$ is called  $\kappa$-Aronszajn, if it has no cofinal branches. The tree property  at an infinite cardinal $\kappa$, denoted $TP(\kappa)$, is the assertion
`` there are no $\kappa$-Aronszajn trees''. There are various results concerning models of the tree property at one or infinitely many cardinals. One of these results that is of our interest in this paper, is a theorem of Cummings and  Foreman \cite{CF}, who produced - relative to the existence of a supercompact cardinal $\kappa$ and a weakly compact cardinal above it - a model in which $\kappa$ is a singular strong limit  cardinal of countable cofinality, and the tree property holds at $\kappa^{++}$. They also stated the the same result for the case $\kappa=\aleph_{\omega+2}$. Recently, Friedman and Halilovi\'c \cite{FHa} obtained the same results, by employing a weaker large cardinal assumption. There are various generalizations of the above-mentioned results, see for example \cite{5A}, \cite{FH}, \cite{FHS1}, and \cite{FHS2}.

Our motivation for this paper is that all papers mentioned in the previous paragraph are merely covering singular cardinals of countable cofinality; and moreover, it is quite natural to ask if the same results can be proved for singular cardinals of uncountable cofinality. In this paper, we start the first step towards answering this question by extending the above cited theorem of Cummings and Foreman
 to singular cardinals of uncountable cofinality, by proving the following theorem:
\begin{theorem}[Main theorem]
Assume that $\kappa$ is a  strong cardinal, and $\lambda > \kappa$ is a measurable cardinal. Suppose that $\delta < \kappa$ is regular. Then there is a generic extension of the universe in which the following hold:
\begin{enumerate}
\item $2^{\kappa}=\kappa^{++}=\lambda$.
\item $\kappa$ is strong limit singular and $\rm{cof}(\kappa)=\delta$.
\item $TP(\kappa^{++})$.
\end{enumerate}
\end{theorem}
The structure of the paper is as follows. In Section 2, we review the definition and the basic properties of Magidor forcing. In Section 3, we present some preliminary results concerning our model. The main forcing construction is then presented in Section 4, where we show that it yields a model in which parts (1) and (2) of Theorem 1.1 are satisfied. Finally, in Section 5, we prove that $TP(\kappa^{++})$ holds in this model as well.
\section{Coherent Sequences and Magidor Forcing}
Magidor forcing for changing cofinality of a large cardinal $\kappa$ to some regular $\delta<\kappa$ was originally defined by Magidor in \cite{Ma} using a Mitchell-increasing sequence of measures over $\kappa$ of length $\delta.$ Subsequentlly, Mitchell \cite{Mi}, defined Radin forcing (of which Magidor forcing is a special case) using coherent sequences of measures. The interested reader could consult \cite{G} for more details.
\begin{definition}[Coherent sequence]
A coherent sequence of measures $\mathcal{U}$ is a function with domain of the form $\{(\alpha,\beta):~\alpha<l^{\mathcal{U}}~\mathrm{and}~\beta<o^{\mathcal{U}}(\alpha)\}$, where $o^{\mathcal{U}}(\alpha)$ is Mitchell order of $\mathcal U$ at $\alpha$ such that the following conditions hold for every $(\alpha,\beta)\in dom(\mathcal U)$.

\begin{enumerate}
\item
$\mathcal{U}(\alpha,\beta)$ is a normal ultrafilter over $\alpha$,
\item
If $j^{\alpha}_{\beta}:V\longrightarrow Ult(V,\mathcal{U}(\alpha,\beta))$ is the canonical embedding, then $$j^{\alpha}_{\beta}(\mathcal{U})\upharpoonright\alpha+1=\mathcal{U}\upharpoonright(\alpha,\beta)$$
\end{enumerate}
where
$$\mathcal{U}\upharpoonright\alpha=\mathcal{U}\upharpoonright\{(\alpha',\beta'):\alpha'<\alpha~\mathrm{and}~\beta'<o^{\mathcal{U}}(\alpha')\}$$
and $$\mathcal{U}\upharpoonright(\alpha,\beta)=\mathcal{U}\upharpoonright\{(\alpha',\beta'):(\alpha'<\alpha~\mathrm{and}~\beta'<o^{\mathcal{U}}(\alpha'))~\mathrm{or}~(\alpha=\alpha'~\mathrm{and}~\beta'<\beta)\}.$$
The ordinal $l^{\mathcal U}$ is called the length of $\mathcal{U}$.
\end{definition}
\begin{definition}
Let $\mathcal{U}=\{\mathcal{U}(\alpha,\beta):\alpha<\kappa+1, \beta<o^{\mathcal{U}}(\alpha)\}$ be a coherent sequence of measures of length $\kappa+1$ and $o^{\mathcal{U}}(\kappa)=\delta$. Define the function $\mathcal{F}:\kappa+1 \longrightarrow V$ by\\
\[  \mathcal{F}(\alpha)= \left\{
\begin{array}{ll}
 \bigcap\limits_{\beta<o^{\mathcal{U}}(\alpha)}\mathcal{U}(\alpha,\beta) &~~~ \text{if}~~~ o^{\mathcal{U}}(\alpha)> 0\\
   & \\
    \{\emptyset\}&  \text{if}~~~o^{\mathcal{U}}(\alpha)=0 \\
\end{array}
\right. \]
\end{definition}
Note that $\mathcal{F}(\alpha)$ is a normal $\alpha$-complete filter over $\alpha$ if $o^{\mathcal{U}}(\alpha)> 0$. We are now ready to define the Magidor forcing.
\begin{definition}
Assume $\mathcal{U}$ is a coherent sequence of measures of length $\kappa+1$ and $o^{\mathcal{U}}(\kappa)=\delta$ is a limit ordinal.
\begin{enumerate}
\item [(a)] The Magidor forcing relative to $\mathcal{U}$, denoted $\mathbb{Q}_{\mathcal{U}},$ consists of finite sequences of the form $p=\langle \langle\alpha_0,A_0\rangle,\dots,\langle\alpha_n,A_n\rangle\rangle$ where:
\begin{enumerate}
\item $\delta < \alpha_0<\dots<\alpha_n=\kappa$,
\item $A_i\in\mathcal{F}(\alpha_i)$,
\item $A_i\cap \alpha_{i-1}=\emptyset$ (where $\alpha_{-1}=\delta+1$).
\end{enumerate}

\item [(b)] Let $p=\langle \langle\alpha_0,A_0\rangle,\dots,\langle\alpha_n,A_n\rangle\rangle$ and $q=\langle \langle\beta_0,B_0\rangle,\dots,\langle\beta_m,B_m\rangle\rangle$ be two conditions in $\mathbb{Q}_{\mathcal{U}}$. We say $q$ is stronger than $p$ ($q\leq p$) if
\begin{enumerate}
\item $m\geq n$,
\item $\forall i\leq n~~\exists j\leq m$ such that $\alpha_i=\beta_j$ and $B_j\subseteq A_i$,
\item $\forall j$ such that $ \beta_j\notin \{\alpha_1,\dots,\alpha_n\},~~B_j\subseteq A_k\cap \beta_j$ and $ \beta_j\in A_k$, where $k$ is the least index such that $\beta_j<\alpha_k$.
\end{enumerate}
\item [(c)]  $q$ is the direct extension (Prikry extension) of $p~(q\leq^* p)$  if
\begin{enumerate}
\item $q\leq p$,
\item $m=n$.
\end{enumerate}
\end{enumerate}
\end{definition}
Given $p \in \MQB_{\mathcal{U}},$ we denote it by
\[
p= \langle  \langle \alpha_0^p, A_0^p  \rangle, \dots,   \langle \alpha_{n^p}^p, A_{n^p}^p\rangle    \rangle,
\]
and call $n^p$ the length of $p$. Note that each $p \in \MQB_{\mathcal{U}}$ can be written as $p=d_p^\frown \langle \alpha_{n^p}^p, A_{n^p}^p\rangle,$
where  $d_p=\langle \langle\alpha_0,A_0\rangle,\dots,\langle\alpha_{n^p-1}^p,A_{n^p-1}^p\rangle\rangle \in V_\kappa$ is the stem of $p$.  The next lemma follows from the fact that any two conditions with the same stem are compatible.
\begin{lemma}
The forcing $(\MQB_{\mathcal{U}}, \leq)$ satisfies the $\kappa^+$-c.c.
\end{lemma}
We also have the following factorization lemma.
\begin{lemma}[ Factorization lemma]
Suppose $p=\langle \langle\alpha_0,A_0\rangle,\dots,\langle\alpha_n,A_n\rangle\rangle \in \MQB_{\mathcal{U}},$ where $n>0$, and also suppose that $m<n.$ Set
\[
p^{\leq m} = \langle \langle\alpha_0,A_0\rangle,\dots,\langle\alpha_m,A_m\rangle\rangle
\]
and
\[
p^{>m}= \langle\langle\alpha_{m+1},A_{m+1}\rangle,\dots,\langle\alpha_n,A_n\rangle\rangle.
\]
Then there is a map
\[
\pi: \MQB_{\mathcal{U}}/ p \longrightarrow \MQB_{\mathcal{U} \upharpoonright \alpha_m+1} / p^{\leq m} \times \MQB_{\mathcal{U}} / p^{>m}
\]
which is a forcing isomorphism with respect to both $\leq$ and $\leq^*.$
\end{lemma}
\begin{proof}
See   \cite{G}, Lemma 5.6.
\end{proof}
We also have the Prikry lemma:
\begin{lemma} [Prikry property]
The forcing $(\MQB_{\mathcal{U}}, \leq, \leq^*)$ satisfies the Prikry property, i.e given any $p \in \MQB_{\mathcal{U}}$ and any statement $\sigma$
in the forcing language of $(\MQB_{\mathcal{U}}, \leq)$, there exists $q \leq^* p$ deciding $\sigma$.
\end{lemma}
\begin{proof}
See \cite{G}, Lemma 5.8.
\end{proof}
Now suppose that $G_{\MQB_{\mathcal{U}}}$ is $\MQB_{\mathcal{U}}$-generic over $V$ and set
\[
C= \{ \beta: \exists p \in G_{\MQB_{\mathcal{U}}}, \text{ and }\exists i< n^p, \text{ such that }\beta=\alpha_i^p          \}.
\]
Then $C$ is a club in $\kappa$ of order type $\delta$, and thus
\[
cof^{V[G_{\MQB_{\mathcal{U}}}]}(\kappa) = cof^V(\delta).
\]
Indeed, $cof^{V[G_{\MQB_{\mathcal{U}}}]}(\kappa) =\delta$ if $\delta$ is regular in $V$. Let $\vec{\beta}=\langle  \beta_i: i < \delta     \rangle$
be an increasing enumeration of $C$. Note that we can recover $G_{\MQB_{\mathcal{U}}}$ from $\vec{\beta}$. To see this, let $G_{\vec{\beta}}$ consists of conditions
$p \in \MQB_{\mathcal{U}}$ such that:
\begin{enumerate}
\item $\{\alpha_0^p, \dots, \alpha_{n^p-1}^p    \} \subseteq \text{range}(\vec{\beta}).$
\item For each $i < \delta$, there exists $q \leq p$ which mentions $\beta_i,$ i.e. $\beta_i=\alpha^q_j$ for some $j < n^q.$
\end{enumerate}
$G_{\vec{\beta}}$ is easily seen to be a filter such that $G_{\MQB_{\mathcal{U}}} \subseteq G_{\vec{\beta}}.$ Hence by the maximality of $G_{\MQB_{\mathcal{U}}}$, we have $G_{\MQB_{\mathcal{U}}}=G_{\vec{\beta}}.$ This convinces us to refer to $\vec{\beta}$ as the Magidor generic sequence (with respect to $\mathcal{U}$).
 The next lemma follows  from Lemmas 2.5 and 2.6.
\begin{lemma}
Assume that $\eta < \kappa$. Let $i< \delta$ be such that $\beta_i \leq \eta < \beta_{i+1}.$ Then
\[
P(\eta) \cap V[\vec{\beta}] = P(\eta) \cap V[\vec{\beta} \upharpoonright i+1].
\]
\end{lemma}
\begin{proof}
Let $p=\langle \langle\alpha_0,A_0\rangle,\dots,\langle\alpha_n,A_n\rangle\rangle \in G_{\vec{\beta}}$ be such that $p$ mentions both $\beta_i$  and $\beta_{i+1}$, say
$\beta_i=\alpha_m$ and $\beta_{i+1}=\alpha_{m+1}$.
 By the Factorization lemma we have
 \[
 \MQB_{\mathcal{U}} / p \simeq \MQB_{\mathcal{U}\upharpoonright \alpha_m+1} /p^{\leq m} \times \MQB_{\mathcal{U}} / p^{>m+1}.
 \]
Let $\dot{A}$ be a $\MQB_{\mathcal{U}}$-name for $A \subseteq \eta$
such that $\Vdash_{\MQB_{\mathcal{U}}}\dot{A} \subseteq \eta$. Let $\dot{B}$ be a $\MQB_{\mathcal{U}} / p^{> m+1}$-name for a subset
of $\MQB_{\mathcal{U}\upharpoonright \alpha_m+1}/p^{\leq m} \times \eta$ such that
\[
\Vdash_{\MQB_{\mathcal{U}} / p^{>m+1}} \forall \gamma < \eta, ((r,f, \gamma) \in
\dot{B}
\iff (r, f) \Vdash_{\MQB_{\mathcal{U}\upharpoonright \alpha_m+1}/p^{\leq m}} \gamma \in \dot{A} ).
\]
Let $\langle y_\alpha: \alpha < \theta < \beta_{i+1}      \rangle$ be an enumeration of $\MQB_{\mathcal{U}\upharpoonright \alpha_m+1}/p^{\leq m} \times \eta$.
Define a $\leq^*$-decreasing sequence $\langle  q_\alpha \mid \alpha < \theta        \rangle$
of conditions in $\MQB_{\mathcal{U}} / p^{>m+1}$ such that for all $\alpha, q_\alpha$ decides ``$y_\alpha \in \dot{B}$''. This is possible as $(\MQB_{\mathcal{U}} / p^{> m+1}, \leq^*)$
is $\beta_{i+1}$-closed, and satisfies the Prikry property. Let $q \leq^* q_\alpha$ for all $\alpha < \theta$.
Then $q$ decides each ``$y_\alpha \in \dot{B}$''. It follows that
 $A\in V[\vec{\beta} \upharpoonright i+1]$ as required.
\end{proof}

It easily follows that the forcing $ \MQB_{\mathcal{U}}$
preserves cardinals.
We need the following theorem of Mitchell \cite{Mi} (see also \cite{Fuchs}, where a characterization is given for the original Magidor forcing).
\begin{theorem}[Characterization theorem]
Assume $V$ is an inner model of $W$. Suppose that $\vec{\beta}=\langle  \beta_i: i < \delta     \rangle \in W$ is an increasing continuous sequence of ordinals. Then $\vec{\beta}$ is $\MQB_{\mathcal{U}}$-generic
over $V$ if and only if the following hold.
\begin{itemize}
\item For every $j<\delta$, $\vec{\beta} \upharpoonright j$ is $\MQB_{\mathcal{U} \upharpoonright j+1}$-generic over $V$.
\item For every $X \in P(\kappa) \cap V$, $X $ is in $\mathcal{F}(\kappa)$ if and only if there exists $j<\delta$ such that $\{ \beta_i: j < i < \delta    \} \subseteq X$.
\end{itemize}
\end{theorem}
\section{Preparation}
Fix a strong cardinal  $\kappa$ and a measurable cardinal $\lambda>\kappa$. We need the following theorem of Woodin.
\begin{theorem}[Woodin, see \cite{GSH}]
Assume $\kappa$ is a strong cardinal. Then there is a forcing notion  of size $\kappa$ such that in the generic extension by it, $\kappa$ remains strong, and its strongness  is indestructible under adding any new Cohen subsets of $\kappa$.
\end{theorem}
By  the above theorem, we may assume in $V$ that the strongness of $\kappa$ is indestructible under $Add(\kappa,\lambda)$. Note that our assumption does not affect measurability of $\lambda$ because it will remain measurable after Woodin's forcing, as the forcing is of size $\kappa < \lambda$. Therefore,  let $D$ be a normal measure on $\lambda$, and let $j:V\longrightarrow M\simeq Ult(V,D)$ be the canonical elementary embedding.
Let
\[
\MPB= Add(\kappa,\lambda).
\]
Suppose that $G_{\MPB}$ is $\MPB$-generic over $V$. Thus by our assumption, $\kappa$ is strong in $V[G_{\MPB}]$, so fix in $V[G_{\MPB}],$ a coherent sequence of measures $\mathcal{U}=\{\mathcal{U}(\alpha,\beta):\alpha\leq\kappa, \beta<o^{\mathcal{U}}(\alpha)\}$ with $o^{\mathcal{U}}(\kappa)=\delta$. Let also $\dot{\mathcal{U}}=\langle\dot{\mathcal{U}}(\alpha,\beta):\alpha\leq\kappa,\beta<o^{\check{\mathcal{U}}}(\alpha)\rangle$ be a $\mathbb{P}$-name for $\mathcal{U}$.
\begin{lemma}
There exists a set $\Lambda\in D$ consisting of Mahlo cardinals such that if  $\xi\in \Lambda$, then
\begin{center}
$\mathcal{U}_{\xi}:=\langle\dot{\mathcal{U}}(\alpha,\beta)^{G}\cap V[G_{\MPB}\upharpoonright\xi]:~\alpha\leq\kappa, \text{ and } \beta<o^{\mathcal{U}}(\alpha)\rangle$
 \end{center}
 is a coherent  sequence of  measures in $V[G_{\MPB}\upharpoonright\xi]= V[G_{\MPB} \cap Add(\kappa, \xi)]$.
\end{lemma}
\begin{proof}
Working in $V$, let $\Lambda$ be the set of all cardinals $\xi < \lambda$
such that
$$
~~\Vdash_{\MPB \upharpoonright \xi} ``\dot{\mathcal{U}}_{\xi}~is~a~coherent~sequence~of~measures\text{''}~,
$$
where $\MPB \upharpoonright \xi = Add(\kappa, \xi).$

 We show that $\Lambda \in D,$
or equivalently, $\lambda \in j(\Lambda).$ We have
$$
j(\Lambda)= \{ \xi < j(\lambda):~~\Vdash^M_{j(\MPB) \upharpoonright \xi} ``j(\dot{\mathcal{U}})_{\xi}~is~a~coherent~sequence~of~measures\text{''}\}.~
$$
Since $\kappa < \lambda$, we have $j(\mathbb{P})\upharpoonright\lambda=\mathbb{P}$; on the other hand
$
\Vdash^M_{\MPB} ``j(\dot{\mathcal{U}})_{\lambda}=\dot{\mathcal{U}}\text{''}.
$
Now the result follows immediately owing to $\mathcal{U}$ was taken to be a coherent sequence of measures in $V[G]$.
\end{proof}
Working in $V[G_{\MPB}]$, let $\mathbb{Q}=\mathbb{Q}_{\mathcal{U}}$ be the Magidor forcing defined using $\mathcal{U}$; and for $\xi \in \Lambda,$  let $\mathbb{Q}_{\xi}=\mathbb{Q}_{\mathcal{U}_\xi}$ be the Magidor forcing defined in $V[G_{\MPB }\upharpoonright\xi]$ relative to $\mathcal{U}_\xi$.
The following lemma follows from Theorem 2.8 (Characterization lemma).
\begin{lemma}
Suppose that $\mathrm{X}$ is a cofinal $\delta$-sequence in $\kappa$ which is $\mathbb{Q}$-generic over $V[G_{\MPB}]$. Then $\mathrm{X}$ is a $\mathbb{Q}_\xi$- generic filter over $V[G_{\MPB} \upharpoonright\xi]$, for all $\xi\in \Lambda$.
\end{lemma}
It follows from the above lemma that for each $\xi\in \Lambda$, there exists a projection
\begin{center}
$\pi_{\xi}:\mathbb{P}\ast\dot{\mathbb{Q}}\longrightarrow RO(\mathbb{P}\upharpoonright\xi\ast\dot{\mathbb{Q}}_{\xi})$,
\end{center}
where $RO(\mathbb{P}\upharpoonright\xi\ast\dot{\mathbb{Q}}_{\xi})$ is the Boolean completion of $\mathbb{P}\upharpoonright\xi\ast\dot{\mathbb{Q}}_{\xi}$.
\begin{lemma}\label{a}
$\mathbb{P}\ast\dot{\mathbb{Q}}$ has the $\kappa^{+}$- Knaster property.
\end{lemma}
 \begin{proof}
Assume that a sequence $\mathcal{A}=\{(p_\alpha,\dot{q}_\alpha):\alpha<\kappa^{+}\} \subseteq \mathbb{P}\ast\dot{\mathbb{Q}}$ is given. Let
 \begin{center}
 $\dot{q}_{\alpha}=(\langle \dot{\beta}^{\alpha}_{0},\dot{A}^{\alpha}_{0}\rangle,\dots$
 $,\langle\langle\dot{\beta}^{\alpha}_{m_{\alpha}-1},\dot{A}^{\alpha}_{m_\alpha-1}\rangle,\langle\kappa,\dot{A}^{\alpha}\rangle)$.
  \end{center}
  Now without loss of generality we can assume that $m_{\alpha}=m=m_{\beta}$ for all $\alpha,\beta<\kappa^+$, also we may assume that there exists a $\mathbb{P}$-name $\dot{d}$ such that for all $\alpha < \kappa^+$, $p_{\alpha}\vdash \dot{d}=(\langle\beta^{\alpha}_{0},\dot{A}^{\alpha}_{0}\rangle,\dots,\langle\beta^{\alpha}_{m_{\alpha}},\dot{A}^{\alpha}_{m_\alpha}\rangle)$; this is because the set $\{(\langle\beta^{\alpha}_{0},\dot{A}^{\alpha}_{0}\rangle,\dots,\langle\beta^{\alpha}_{m_{\alpha}},\dot{A}^{\alpha}_{m_\alpha}\rangle,\langle\kappa,\dot{A}_{\alpha}\rangle):\alpha<\kappa^+\}$ has size $\kappa$.
Thus let $\mathcal{A}=\{(p_\alpha,(\dot{d}^{\frown}\langle (\kappa,\dot{A}_{\alpha})\rangle):\alpha<\kappa^+\}$.  As $\MPB$ has the $\kappa^+$-Knaster property, there exists $I\subseteq\kappa^+$ unbounded such that $\langle p_{\alpha}:\alpha \in I\rangle$ consists of  pairwise compatible conditions. This concludes the lemma,
as then any two conditions in $\mathcal{A} \upharpoonright I$ are compatible.
 \end{proof}
\section{Main Forcing Notion}
We are now ready to define our main forcing notion. Thus fix a $\mathbb{P}=Add(\kappa, \lambda)$-generic filter $G_{\MPB}$ over $V$, and let $\mathcal{U} \in V[G_{\MPB}]$ be a coherent sequence of measures
 of length $\kappa+1$ such that $o^{\mathcal{U}}(\kappa) =\delta$. Let $\MQB=\MQB_{\mathcal{U}}$, and let $\mathrm{X}$ be a  $\delta$-sequence cofinal in $\kappa$, which is $\MQB$-generic over $V[G_{\MPB}]$. Fix $\Lambda \in D$ as before. Note that for each $\xi\in \Lambda$, $V[G_{\MPB} \upharpoonright\xi][\mathrm{X}]\subseteq V[G_{\MPB}][\mathrm{X}]$. Let  $\mathbb{R}_{\xi}=Add(\kappa^{+},1)_{V[G_{\MPB}\upharpoonright\xi][\mathrm{X}]}$, for each $\xi\in \Lambda$.
\begin{definition}[Main Forcing]\
\begin{itemize}
\item [(a)] Conditions in $\mathbb{R}$ are  triples $(p,\dot{q},r)$ such that:
\begin{enumerate}
\item $(p,\dot{q})\in \mathbb{P}\ast\dot{\mathbb{Q}}$,
\item $r$ is a partial function with $dom(r)\subseteq \Lambda$ and $|dom(r)|\leq\kappa$,
\item For every $\xi\in dom(r),~r(\xi)$ is a $\mathbb{P}\upharpoonright\xi\ast\dot{\mathbb{Q}}_{\xi}$-name for a condition in $\mathbb{R}_{\xi}$.
\end{enumerate}
\item [(b)] For conditions $(p_0,\dot{q}_0,r_0)$ and $(p_1,\dot{q}_1,r_1)$ in $\mathbb{R}$, we say $(p_1,\dot{q}_1,r_1)\leq (p_0,\dot{q}_0,r_0)$ iff
\begin{enumerate}
\item [(4)]
$(p_1,\dot{q}_1)\leq(p_0,\dot{q}_0)$ in $\mathbb{P}\ast\dot{\mathbb{Q}}$,
\item [(5)]
$dom(r_0)\subseteq dom(r_1)$ and for all $\xi\in dom(r_0)$; $\pi_{\xi}(p_1,\dot{q}_1)\vdash$`` $ r_1(\xi)\leq r_0(\xi)$''.
\end{enumerate}
\end{itemize}
\end{definition}
\begin{lemma}\label{b}
Let $\mathbb{U}=\{(1_{\MPB}, 1_{\MQB}, r): (1_{\MPB}, 1_{\MQB}, r)\in\mathbb{R}\}$, i.e. the set of third coordinates of $\mathbb{R}$. Consider the function $\rho:(\mathbb{P}\ast\dot{\mathbb{Q}})\times\mathbb{U}\longrightarrow\mathbb{R}$ defined by $\rho(\langle(p, \dot{q}), (1_{\MPB}, 1_{\MQB}, r)\rangle)=(p, \dot{q}, r)$. Then
\begin{enumerate}
\item$\mathbb{U}$ is $\kappa^{+}$-closed.
\item$\rho$ is a projection, and
$$V^{\mathbb{P}\ast\dot{\mathbb{Q}}} \subseteq V^{\mathbb{R}}\subseteq V^{(\mathbb{P}\ast\dot{\mathbb{Q}})\times\mathbb{U}}.$$
\end{enumerate}
\end{lemma}

\begin{proof}
\begin{enumerate}
\item
Let $\{(1_{\MPB}, 1_{\MQB}, r_{\xi}):\xi<\kappa\}$ be a decreasing sequence of conditions in $\mathbb{U}$. Then the sequence $\langle dom(r_{\xi}):\xi<\kappa\rangle$ is $\subseteq$-increasing. Let $r$ be a function with $dom(r)=\bigcup\limits_{\xi<\kappa}dom(r_{\xi})$; we are going to define $r$ on this set. If $\alpha\in dom(r)$, then there exists $\xi_{\alpha}$ such that $\alpha\in dom(r_{\xi_{\alpha}})$, and then $\alpha\in dom(r_{\xi})$ for all $\xi\geq\xi_{\alpha}$. We have $$1_{\mathbb{P}\upharpoonright\xi\ast\mathbb{Q}_\xi}\vdash ``\langle r_{\xi}(\alpha):\xi\geq\xi_{\alpha}\rangle~is~a~decreasing~sequence~in~Add(\kappa^{+},1)".$$
 On the other hand $1_{\mathbb{P}\upharpoonright\xi\ast\mathbb{Q}_\xi}\vdash ``Add(\kappa^+,1)~is~\kappa^+\text{-}closed$ '', thus there exists (by maximal completeness) a name $\dot{\tau}_\alpha$ such that $1_{\mathbb{P}\upharpoonright\xi\ast\mathbb{Q}_\xi}$ forces it to be a lower bound for the above-mentioned sequence. Let $r(\alpha)=\dot{\tau}_\alpha$. Now $(1_{\MPB}, 1_{\MQB}, r)$ is a lower bound for $\{(1_{\MPB}, 1_{\MQB}, r_{\xi}):\xi<\kappa\}$.
\item
Clearly, $\rho$ preserves ordering and $\rho(1)=1$. It remains to show that if $(p_1, \dot{q}_1, r_1)\leq \rho(\langle (p_0, \dot{q}_0), (1_{\MPB}, 1_{\MQB}, r_0) \rangle)=(p_0, \dot{q}_0, r_0)$, then there exists
\[
\langle (p_2, \dot{q}_2), (1_{\MPB}, 1_{\MQB}, r_2) \rangle \leq \langle (p_0, \dot{q}_0), (1_{\MPB}, 1_{\MQB}, r_0) \rangle
\]
 such that $\rho(\langle (p_2, \dot{q}_2), (1_{\MPB}, 1_{\MQB}, r_2) \rangle)= (p_2, \dot{q}_2, r_2) \leq (p_1, \dot{q}_1, r_1)$.
 Let $(p_2, \dot{q}_2)=(p_1, \dot{q}_1)$. Put $dom(r_2):=dom(r_1)$. We define $r_2(\xi)$, for  $\xi\in dom(r_2)$, such that
 $(p_2, \dot{q}_2, r_2) \leq (p_1, \dot{q}_1, r_1)$ and  for $\xi\in dom(r_0)$, $1_{\mathbb{P}\upharpoonright\xi\ast\mathbb{Q}_\xi}\vdash r_2(\xi)\leq r_{0}(\xi)$.\\
By maximal completeness,  there exists a name $\dot{\tau}_\xi$ such that $\pi_\xi(p_1, \dot{q}_1) \vdash\dot{\tau}_\xi=\dot{r}_1(\xi)$
and
$(p^*, \dot{q}^*)\vdash \dot{\tau}_\xi=\dot{q}_0(\xi)$, for all $(p^*, \dot{q}^*)\in RO(\MPB\upharpoonright \xi \ast \dot{\MQB}_\xi)$ with
$(p^*, \dot{q}^*)\perp \pi_\xi(p_1, \dot{q}_1)$.
Set $r_2(\xi)=\dot{\tau}_\xi$, which concludes that $\rho$ is a projection.\\
It is also obvious that $\mathbb{P}\ast\dot{\mathbb{Q}}$ is a projection of $\mathbb{R}$, and thus $V^{\mathbb{P}\ast\dot{\mathbb{Q}}} \subseteq V^{\mathbb{R}}\subseteq V^{(\mathbb{P}\ast\dot{\mathbb{Q}})\times\mathbb{U}}.$
\end{enumerate}
\end{proof}
\begin{corollary}
 $V^{\mathbb{R}}$ and $V^{\mathbb{P}\ast\dot{\mathbb{Q}}}$ have the same $\kappa$-sequences of ordinals.
\end{corollary}
\begin{proof}
Suppose $f:\kappa\longrightarrow Ord$ is a $\kappa$-sequnece in $V^{\mathbb{R}}$.  Then by Lemma 4.2(2), $f\in V^{(\mathbb{P}\ast\dot{\mathbb{Q}})\times\mathbb{U}}$. By
Lemma 4.2(1),  $\mathbb{U}$ is $\kappa^{+}$-closed, and hence  $f\in V^{\mathbb{P}\ast\dot{\mathbb{Q}}}$, which completes the proof.
\end{proof}
Before we continue, let us recall Easton's lemma.
\begin{lemma}[Easton's lemma] Assume $\kappa$ is an infinite cardinal, $\mathbb{P}$
is a $\kappa^+$-c.c. forcing, and $\mathbb{Q}$ is a $\kappa^+$-closed forcing. If $G \times H$ is $\mathbb{P} \times \mathbb{Q}$-generic over $V$, then
\begin{enumerate}
\item $\Vdash_{\mathbb{Q}}$``$\mathbb{P}$ is $\kappa^+$-c.c.''.
\item $\Vdash_{\mathbb{P}}$``$\mathbb{Q}$ is $\kappa^+$-distributive''.
\item If $f: \kappa \to V$ and $f \in V[G \times H],$ then $f \in V[G].$
\end{enumerate}
\end{lemma}
\begin{corollary}
 Let $G_{\MRB}$ be an $\mathbb{R}$-generic filter over $V$ and let $G_{\mathbb{P}\ast\dot{\mathbb{Q}}}$ be the $\mathbb{P}\ast\dot{\mathbb{Q}}$-generic filter induced by $G_{\MRB}$. If $E$ is a set of ordinals of size $\kappa$, then $E\in V[G_{\mathbb{P}\ast\dot{\mathbb{Q}}}]$.
\end{corollary}
\begin{proof}
Let $H$, a $\mathbb{U}$-generic filter, be such that $V[G_{\MRB}]\subseteq V[G_{\mathbb{P}\ast\dot{\mathbb{Q}}}][H]$. As $\mathbb{P}\ast\dot{\mathbb{Q}}$ is $\kappa^+$-c.c., and $\mathbb{U}$ is $\kappa^+$-closed, Easton's lemma ensures us that $E\in V[G_{\mathbb{P}\ast\dot{\mathbb{Q}}}]$ as required.
\end{proof}

\begin{lemma}
$\mathbb{R}$  has the following properties:
\begin{enumerate}
\item $\mathbb{R}$  is $\lambda$-Knaster, hence it preserves cardinals $\geq\lambda$.
\item $\mathbb{R}$ preserves $\kappa^+$.

\item $V^{\mathbb{R}}\models 2^{\kappa}=\lambda$.

\item $\mathbb{R}$ preserves cardinals below $\kappa^+$.

\item  $\mathbb{R}$ collapses the cardinals in the inteval $(\kappa^{+},\lambda)$ onto $\kappa^{+}$, so $\lambda=\kappa^{++}$.
\end{enumerate}
\end{lemma}
\begin{proof}\
\begin{enumerate}
\item Suppose that $\langle(p_\alpha,\dot{q}_\alpha,r_\alpha):~\alpha<\lambda\rangle \subseteq \mathbb{R}$ is given. By refining the sequence and using Lemma \ref{a}, we may assume that $\langle(p_\alpha, \dot{q}_\alpha):\alpha<\lambda\rangle$  consists of  pairwise compatible conditions. Then using this fact and counting of the possible names as values of $r_\alpha$, one can refine the sequence further to obtain a  subsequent of size $\lambda$ consisting of pairwise compatible conditions.
\item Suppose $f:\kappa\longrightarrow \kappa^+$ is in $V[G_{\MRB}]$ for an $\mathbb{R}$-generic filter $G_{\MRB}$.  Then $f\in V[G_{\mathbb{P}\ast\dot{\mathbb{Q}}}]$ by Corollary 4.4. On the other hand the forcing $\mathbb{P}\ast\dot{\mathbb{Q}}$ is  $\kappa^+$-c.c., so $f$ is bounded. As a consequence $\kappa^+$ is not collapsed.

\item We have $P(\kappa)\cap V[G_{\MRB}]=P(\kappa)\cap V[G_{\mathbb{P}\ast\dot{\mathbb{Q}}}]$, which easily implies $2^{\kappa}=\lambda$.

\item it is obvious as $\mathbb{P}\ast\dot{\mathbb{Q}}$ preserves cardinals and $\mathbb{U}$ is $\kappa^+$-closed.
\item Let $\kappa^+ < \xi < \lambda$. It is easy to see that the function $\sigma_\xi:\mathbb{R}\longrightarrow Add(\kappa,\xi)\ast A\dot{d}d(\kappa^+,1)$ defined by $\sigma_\xi(p,\dot{q},r)=(p \upharpoonright \xi, r(\xi))$ is a projection ( if $\xi\notin dom(r)$, set $r(\xi)=\check{\emptyset}$); but then $Add(\kappa,\xi)\ast A\dot{d}d(\kappa^+,1)$ collapses $\xi$ onto $\kappa^+$, thus forcing with $\mathbb{R}$ does the collapsing.
\end{enumerate}

\end{proof}

It follows from the above results that in every generic extention of $V$ by $\mathbb{R}$,  items 1 and 2 of   Theorem 1.1 are valid. It remains to show that the tree property also holds at $\kappa^{++}$. The rest of the paper is devoted to prove this fact.
 \section{The tree Property in $V^{\mathbb{R}}$}
Let $G_{\MRB}$ be $\mathbb{R}$-generic over $V$ and let $T\in V[G_{\MRB}]$ be a $\lambda$-tree. We show that $T$ has a cofinal branch in $V[G_{\MRB}]$.
 Recall that in $V$, we have fixed an elementary embedding $j: V \to M \simeq \Ult(V, D)$, where $D$ is a normal measure on $\lambda.$ 
 We have
 \[
 j(\mathbb{P})=Add(\kappa, j(\lambda)) = Add(\kappa, \lambda) \times Add(\kappa, [\lambda, j(\lambda))) = \mathbb{P} \times Add(\kappa, [\lambda, j(\lambda))).
 \]
 Thus, working in $V[G_{j(\mathbb{P})}],$
 we can extend $j$ to some $j: V[G_{\mathbb{P}}] \to M[j(G_{\mathbb{P}})]$.
 Assume $X$ is $\mathbb{Q}^M_{j(\mathcal{U})}$-generic over $V[G_{j(\mathbb{P})}].$ Then thanks to the characterization lemma,
$X$ is   $\mathbb{Q}$-generic over $V[G_{\mathbb{P}}]$ as well. Note that clearly   $j(\mathrm{X})=\mathrm{X}$ holds.
 Consider $j(\mathbb{R})$ whose conditions are triples $(p,\dot{q},r)$ satisfying the following:
\begin{enumerate}
\item
$(p,\dot{q})\in j(\MPB)\ast\dot{\mathbb{Q}}^{M}_{j(\mathcal{U})}$,
\item $r$ is a parial function with $dom(r)\subseteq j(\Lambda)$ and $|\dom(r)|\leq\kappa$,
\item For every $\xi\in \dom(r)$, $~r(\xi)$ is a $j(\mathbb{P})\upharpoonright\xi\ast\dot{\mathbb{Q}}^{M}_{j(\mathcal{U})_\xi}$-name for a condition in $j(\mathbb{R})_{\xi}=Add(\kappa^{+},1)_{V[G_{j(\mathbb{P}) \upharpoonright \xi}][\mathrm{X}]}$.
 \end{enumerate}
 It is easily seen that we have a projection from $j(\MRB)$ onto $\MRB,$ that we call it $\varrho.$

 \begin{lemma}
$T$ has a cofinal branch in $V[j(\mathbb{R})]$.
\begin{proof}
Recall from the above that we have a projection $\varrho: j(\MRB) \to \MRB$. Hence one can find $G_{j(\MRB)}$ such that $G_{j(\MRB)}$  is $j(\MRB)$-generic over $V$ and $\varrho[G_{j(\MRB)}]=G_{\MRB}$. Then we can lift $ j$ to some $\bar j: V[G_{\MRB}] \to M[G_{j(\MRB)}]$,
which is defined in $V[G_{j(\MRB)}].$
Consider $\bar j(T) \in M[G_{j(\MRB)}]$ that is a $j(\lambda)$-tree with $\bar j(T)\upharpoonright\lambda=T$, and since $j(\lambda) > \lambda$, we can take a node $t^*$ in the $\lambda$-th level of $\bar j(T)$, then
\[
b=\{t \in \bar j(T): t <_{\bar j(T)} t^*       \} = \{ t \in T: t <_{\bar j(T)} t^*     \}
\]
forms a cofinal branch of $T$ in $M[G_{j(\MRB)}]\subseteq V[G_{j(\MRB)}]$.
\end{proof}
\end{lemma}
Let $b$ be a cofinal branch in $V[G_{j(\MRB)}]$. We will show that $b\in V[G_{\MRB}]$. We will do this by showing that passing from $V[G_{\MRB}]$ to $V[G_{j(\MRB)}]$ by $j(\MRB)/ G_{\MRB}$ does not add a cofinal branch through $T$. For this we need a careful analysis of the quotient forcing $j(\mathbb{R})/G_{\MRB}$.
\begin{lemma}
In $V[G_{\MRB}]$, there exists a $\kappa^{+}$-closed forcing notion $\mathbb{U}^{*}$ such that there is a projection $\pi$ from $((j(\mathbb{P})\ast\dot{\mathbb{Q}}^{M}_{j(\mathcal{U})})/G_{\mathbb{P}\ast\dot{\mathbb{Q}}})\times\mathbb{U}^{*}$
onto $j(\mathbb{R})/G_{\mathbb{R}}$ (where $G_{\mathbb{P}\ast\dot{\mathbb{Q}}}$ is obtained from $G_{\MRB}$ in the natural way).
\end{lemma}
\begin{proof}
Let $\mathbb{U}^{*}=\{(p,q,r)\in j(\mathbb{R})/G_{\MRB}:p=q=1\}$, and define
\begin{center}
$\pi:((j(\mathbb{P})\ast\dot{\mathbb{Q}}^{M}_{j(\mathcal{U})})/G_{\mathbb{P}\ast\dot{\mathbb{Q}}})\times\mathbb{U}^{*}\longrightarrow j(\mathbb{R})/G_{\mathbb{R}}$
 \end{center}
 by $\pi(\langle (p,q),(0,0,r)\rangle)=(p,q,r)$.
 We show that  $\mathbb{U}^*$ and $\pi$ are as required.

Let us show that $\mathbb{U}^*$ is $\kappa^+$-closed; thus assume $\langle  (p_\alpha, q_\alpha, r_\alpha): \alpha < \kappa         \rangle$
 is a decreasing sequence of conditions in $\mathbb{U}^{*}$, then for all $\alpha < \kappa, p_\alpha=q_\alpha = 1$ and
 $(p_\alpha, q_\alpha, r_\alpha) \in j(\mathbb{R})/G_{\MRB}.$ Thus we have the following:
 \begin{enumerate}
\item If $\alpha < \kappa$, then  $(1, 1, r_\alpha) \in j(\mathbb{R})$.
\item If $\alpha < \kappa,$ then $\varrho(1, 1, r_\alpha) =(1, 1, r^*_\alpha) \in G_{\MRB},$
where $r^*_\alpha$ is defined as follows:
\begin{enumerate}
\item $r^*_\alpha$ is a parial function with $\dom(r^*_\alpha)=\dom(r) \cap \Lambda \subseteq \Lambda$.
\item For every $\xi\in \dom(r^*_\alpha)$, $~r^*_\alpha(\xi)=r_\alpha(\xi)$ is a $\mathbb{P}\upharpoonright\xi\ast\dot{\mathbb{Q}}_{\mathcal{U}_\xi}$-name for a condition in $\mathbb{R}_{\xi}=Add(\kappa^{+},1)_{V[G_{\mathbb{P} \upharpoonright \xi}][\mathrm{X}]}$.
\end{enumerate}
 \end{enumerate}
 Let $r$ be a function with $\dom(r)=\bigcup_{\xi<\kappa}\dom(r_\alpha)$. For $\xi \in \dom(r)$ we define $r(\xi)$ as follows.
 Let  $\alpha_{\xi}$ be such that $\xi\in \dom(r_{\alpha_{\xi}})$. Then $\xi\in \dom(r_{\alpha})$ for all $\alpha\geq\alpha_{\xi}$. We have $$1_{j(\mathbb{P})\upharpoonright\xi\ast\dot{\mathbb{Q}}^{M}_{j(\mathcal{U})_\xi}}\vdash ``\langle r_{\alpha}(\xi):\alpha\geq\alpha_{\xi}\rangle~is~a~decreasing~sequence~in~Add(\kappa^{+},1)".$$
 On the other hand $1_{j(\mathbb{P})\upharpoonright\xi\ast\dot{\mathbb{Q}}^{M}_{j(\mathcal{U})_\xi}}\vdash ``Add(\kappa^+,1)~is~\kappa^+\text{-}closed$ '', thus there exists (by maximal completeness) a name $\dot{\tau}_\xi$ such that $1_{j(\mathbb{P})\upharpoonright\xi\ast\dot{\mathbb{Q}}^{M}_{j(\mathcal{U})_\xi}}$ forces it to be the greatest lower bound of the above-mentioned sequence, i.e.
$$1_{j(\mathbb{P})\upharpoonright\xi\ast\dot{\mathbb{Q}}^{M}_{j(\mathcal{U})_\xi}}\Vdash \dot{\tau}_\xi=\bigwedge_{\alpha \geq \alpha_\xi}r_{\alpha}(\xi).$$
  Let $r(\xi)=\dot{\tau}_\xi$. It is evident that $(1, 1, r) \in j(\mathbb{R})$, and that it is a lower bound for $\{(1, 1, r_{\alpha}):\alpha<\kappa\}$. It remains to show that $(1, 1, r) \in \mathbb{U}^*$, or equivalently  $\varrho(1, 1, r) \in G_{\MRB}.$ Let
 $\varrho(1, 1, r)=(1,1, r^*).$ Then $\dom(r^*)=\dom(r) \cap \Lambda = \bigcup_{\xi<\kappa}(\dom(r_\alpha) \cap \Lambda) = \bigcup_{\xi<\kappa}\dom(r^*_\alpha)$, and for all $\xi \in \dom(r^*)$ we have
 $$1_{\mathbb{P}\upharpoonright\xi\ast\dot{\mathbb{Q}}_{\mathcal{U}_\xi}}\Vdash r^*(\xi)=\dot{\tau}_\xi=\bigwedge_{\alpha \geq \alpha_\xi}r^*_{\alpha}(\xi).$$
 Now we are done as clearly $(1,1, r^*) \in G_{\MRB}$.

It is also obvious that $\pi$ is a projection and the result follows.
\end{proof}
\begin{lemma}\label{QQ}
 $V^{\mathbb{R}\ast\dot{\mathbb{U}}^*}\models$``$(j(\mathbb{P})\ast\dot{\mathbb{Q}}^{M}_{j(\mathcal{U})}) / G_{\MPB \ast \dot{\MQB}}$ has the $\kappa^{+}$-Knaster property''.
\end{lemma}
To prove the lemma, we need a finer analysis of the  quotient forcing $(j(\mathbb{P})\ast\dot{\mathbb{Q}}^{M}_{j(\mathcal{U})})/ G_{\mathbb{P}\ast\dot{\mathbb{Q}}}$.
\begin{claim}
Assume  $p^{*}=\langle p, \langle (\alpha_1,\dot{A}_1),\dots,(\alpha_{n-1},\dot{A}_{n-1}),(\alpha_n=\kappa,\dot{A})\rangle\rangle\in\mathbb{P}\ast\dot{\mathbb{Q}}$ and $q^{*}=\langle q,\langle(\beta_1,\dot{B}_1),\dots,(\beta_{m-1},\dot{B}_{m-1}),(\beta_m=\kappa,\dot{B})\rangle\rangle\in j(\mathbb{P})\ast\dot{\mathbb{Q}}^{M}_{j(\mathcal{U})}$. Then
\[
p^{*}\vdash_{\MPB \ast \dot{\MQB}} \text{~``~} q^{*}\notin (j(\mathbb{P})\ast\dot{\mathbb{Q}}^{M}_{j(\mathcal{U})})/G_{\mathbb{P}\ast\dot{\mathbb{Q}}} \text{~''}
\]
if and only if one of the following hold:
\begin{enumerate}
\item
$p\perp q\upharpoonright\lambda.$
\item
$p\parallel q\upharpoonright\lambda$ and there exists $j$ such that $\beta_j\notin\{\alpha_0,\dots,\alpha_n\},$ and $p\cup q\vdash_{j(\MPB)}$``$\beta_j\notin \dot{A}_{k}$'', where $k$ is the least index such that $\beta_j<\alpha_k$.
\item
$p\parallel q\upharpoonright\lambda$ and there exists $j$ such that $\beta_j\notin\{\alpha_0,\dots,\alpha_n\},$ and $p\cup q\vdash_{j(\MPB)}$``$ \dot{B}_j\nsubseteq \dot{A}_{k}\cap\beta_j$'',  where $k$ is the least index such that $\beta_j<\alpha_k$.
\item
$p\parallel q\upharpoonright\lambda$ and there exists $i$ such that $\alpha_i\notin\{\beta_0,\dots,\beta_m\},$ and $p\cup q\vdash_{j(\MPB)}$``$\alpha_i\notin \dot{B}_{k}$'',  where $k$ is the least index such that $\alpha_i<\beta_k$.

\item
$p\parallel q\upharpoonright\lambda$ and there exists $i$ such that $\alpha_i\notin\{\beta_0,\dots,\beta_m\},$ and $p\cup q\vdash_{j(\MPB)}$``$ \dot{A}_i\nsubseteq\dot{B}_{k}\cap\alpha_i$'', where $k$ is the least index such that $\alpha_i<\beta_k$.
\end{enumerate}
\end{claim}
\begin{proof}
If one of the clauses (1) to (5) holds, then it is clear that
 \[
p^{*}\vdash_{\MPB \ast \dot{\MQB}} \text{~``~} q^{*}\notin (j(\mathbb{P})\ast\dot{\mathbb{Q}}^{M}_{j(\mathcal{U})})/G_{\mathbb{P}\ast\dot{\mathbb{Q}}} \text{~''}.
\]
This is because if any of the cases (1) - (5) holds, then we cannot simultaneously have  $p^{*} \in G_{\mathbb{P}\ast\dot{\mathbb{Q}}}$ and $\pi(q^{*})\in G_{\mathbb{P}\ast\dot{\mathbb{Q}}}$.

On the other hand assume all conditions (1)-(5) fail. Then $p$ is compatible with $q\upharpoonright\lambda$, and so $p \cup q \in j(\MPB)$ is well-defined.
 Set
 \begin{center}
 $p^*(1)=\langle (\alpha_1,\dot{A}_1),\dots,(\alpha_{n-1},\dot{A}_{n-1}),(\alpha_n=\kappa,\dot{A})\rangle$
  \end{center}
 and
  \begin{center}
  $q^*(1)=\langle(\beta_1,\dot{B}_1),\dots,(\beta_{m-1},\dot{B}_{m-1}),(\beta_m=\kappa,\dot{B})\rangle$.
 \end{center}
 We show that $p \cup q \Vdash$``$p^*(1)$ and $q^*(1)$ are compatible''.

 Let $\{\gamma_1, \dots, \gamma_t\}$ be an increasing enumeration of $\{\alpha_o,\dots, \alpha_n\}\cup\{\beta_o,\dots, \beta_m\}$. Suppose $k \leq t$. We would like to define a name
 $\dot{C}_k.$ There are several cases to consider:
  \begin{enumerate}
   \item There exist $i$ and $j$ such that $\alpha_i=\gamma_k=\beta_j$, then let $\dot{C}_k$ be a name forced by $p\cup q$ to be $\dot{A}_i\cap \dot{B}_j$.
   \item There exists $i$ such that $\gamma_k=\alpha_i\notin\{\beta_1,\dots, \beta_m\}$. Let $j$ be the least index such that $\gamma_k < \beta_j$
   and let $\dot{C}_k$ be a name forced by $p\cup q$ to be $\dot{A}_i\cap (\dot{B}_j \cap \alpha_i)$.
  \item  There exists $j$ such that $\gamma_k=\beta_j\notin\{\alpha_0,\dots, \alpha_n\}$. Let $i$ be the least index such that $\gamma_k < \alpha_i$
   and let $\dot{C}_k$ be a name forced by $p\cup q$ to be $(\dot{A}_i \cap \beta_j) \cap \dot{B}_j$.
\end{enumerate}
Let
\begin{center}
$r^*=\langle p\cup q,  \langle(\gamma_1,\dot{C}_1), \dots, (\gamma_t=\kappa, \dot{C}_t)\rangle \rangle$.
\end{center}
As none of the clauses (1)-(5) hold, we can easily check that $r^* \in j(\mathbb{P})\ast\dot{\mathbb{Q}}^{M}_{j(\mathcal{U})}$ is a well-defined condition, and it extends both $p^*$ and $q^*.$ Let $G_{j(\mathbb{P})\ast\dot{\mathbb{Q}}^{M}_{j(\mathcal{U})}}$ be a $j(\mathbb{P})\ast\dot{\mathbb{Q}}^{M}_{j(\mathcal{U})}$-generic filter containing $r^*$, and such that its projection to $\MPB \ast \dot{\MQB}$ gives $G_{\MPB \ast \dot{\MQB}}$.
Then
\[
p^{*}\vdash_{\MPB \ast \dot{\MQB}} \text{~``~} r^{*}\in (j(\mathbb{P})\ast\dot{\mathbb{Q}}^{M}_{j(\mathcal{U})})/G_{\mathbb{P}\ast\dot{\mathbb{Q}}} \text{~ and~}r^* \leq q^*\text{~''}.
\]
Thus
\[
p^{*}\vdash_{\MPB \ast \dot{\MQB}} \text{~``~} q^{*}\in (j(\mathbb{P})\ast\dot{\mathbb{Q}}^{M}_{j(\mathcal{U})})/G_{\mathbb{P}\ast\dot{\mathbb{Q}}} \text{~''}
\]
which gives the result.
\end{proof}
For the proof of  Lemma \ref{QQ}, we also need the following lemma from \cite{CF}.
\begin{lemma}
Let $V\subseteq W$ be two inner models of $ZFC$ and let $\kappa$ be a limit cardinal in $W$. Suppose that the following properties hold:
\begin{enumerate}
\item $V\models \kappa=\kappa^{<\kappa}$.
\item $W$ computes $\kappa^+$ correctly.
\item Every set of ordinals of size at most $\kappa$ in $W$ is covered by a set of size at most $\kappa$ in $V$ .
\end{enumerate}
Let $\langle x_{\alpha}:\alpha<\kappa^+\rangle$ be a $\kappa^+$-sequence of sets of ordinals such that $x_{\alpha}\in V$, and $|x_\alpha|<\kappa$ for all $\alpha<\kappa^+$. Then there exists $I\subseteq\kappa^+$ unbounded such that $\langle x_{\alpha}:\alpha\in I\rangle$ forms a $\Delta$-system.
\end{lemma}

\begin{proof}[Proof of Lemma \ref{QQ}]\
Assume that the lemma fails. Thus let $((p, \dot{q}, r), \dot{u})\in \mathbb{R}\ast\dot{\mathbb{U}}^*$ be a condition that forces ``the sequence $\langle \dot{w}_{\alpha}:~\alpha<\kappa^+\rangle$  witnesses that   $(j(\mathbb{P})\ast\dot{\mathbb{Q}}^{M}_{j(\mathcal{U})})/G_{\mathbb{P}\ast\dot{\mathbb{Q}}}$ is not
$\kappa^+$-Knaster, where $\dot{w}_\alpha$ is of the form $(p_\alpha,\dot{q}_\alpha)$''. By Lemma 5.5, there exists $I_1\subseteq\kappa^+$
 of size $\kappa^+$ such that $((p, \dot{q}, r), \dot{u})$ forces ``$\langle p_\alpha:\alpha\in I_1 \rangle$ consists of  pairwise compatible
 conditions''.

Assume $\dot{q}_\alpha$ is forced to be $\dot{d}_{\alpha}^{‎\frown}‎\langle\langle\kappa, \dot{A}_\alpha\rangle\rangle$, where $\dot{d}_{\alpha}=\langle\langle\gamma^\alpha_0, \dot{A}^\alpha_0\rangle, \dots,\langle\gamma^\alpha_{m_\alpha-1}, \dot{A}^\alpha_{m_\alpha-1}\rangle\rangle$. Let $I_2 \subseteq I_1$
be of size $\kappa^+$ such that  $p_\alpha\vdash \dot{d}_\alpha=\dot{d}$, for some fixed $\dot{d} \in V_\kappa$, and all $\alpha \in I_2$. Therefore, for every $\alpha \in I_2$, $p_\alpha \Vdash$``$\dot{q}_\alpha=\langle\dot{d},\langle\kappa, \dot{A}_\alpha\rangle\rangle$''. Let $\dot{q}=\dot{d}_q^\frown\langle\langle\kappa, \dot{B}\rangle\rangle$. By strengthening $((p, \dot{q}, r), u)$, if necessary, we can assume that $p$  decides $\dot{d}$ which is forced to be the lower part of each $\dot{q}_\alpha$ such that $\alpha \in I_2$.

Let $G_{\MRB}\ast G_{\mathbb{U}^*}$ be an $\mathbb{R}\ast\mathbb{U}^*$-generic filter containing $((p,\dot{q},r), u)$. Then $(p,\dot{q},r)\in G_{\MRB}$ and hence $(p,\dot{q})\in G_{\mathbb{P}\ast\dot{\mathbb{Q}}}$. Similarly $(p_\alpha\upharpoonright\lambda,\dot{q_\alpha})\in G_{\mathbb{P}\ast\dot{\mathbb{Q}}}$. Thus
there exists $(p_\alpha^*, \dot{q}_{\alpha}^*)$ in $G_{\mathbb{P}\ast\dot{\mathbb{Q}}}$ extending both $(p,\dot{q})$ and
$(p_\alpha\upharpoonright\lambda,\dot{q_\alpha})$. As before, write $\dot{q}^{*}_{\alpha}$ as $\dot{d}^*_\alpha$$^\frown\langle\langle\kappa,\dot{B}_\alpha\rangle\rangle$. Now there exists an unbounded subset $I_3\subseteq I_2$
 and some $\dot{d}^* \in V_\kappa$ such that for all $\alpha \in I_3,$ $p_\alpha^*\vdash \dot{d^*_\alpha}=d^*$. Now consider $((p^*,\dot{q}^*,r),u)$, where $q^*=d^{*}$$^\frown\langle\langle\kappa,\dot{B}\rangle\rangle$.

Let $\alpha, \beta\in I_2$. Then by the construction we have
 $$(p^*,\dot{q}^*)\vdash``(p_\alpha,\dot{q}_\alpha)\in (j(\mathbb{P})\ast\dot{\mathbb{Q}}^{M}_{j(\mathcal{U})})/(\mathbb{P}\ast\dot{\mathbb{Q}})\text{''}$$
  and
   $$(p^*,\dot{q}^*)\vdash``(p_\beta,\dot{q}_\beta)\in (j(\mathbb{P})\ast\dot{\mathbb{Q}}^{M}_{j(\mathcal{U})})/(\mathbb{P}\ast\dot{\mathbb{Q}})\text{''}.$$
  Claim 5.4 now ensures us that
 $$(p^*,\dot{q}^*)\vdash (p_\alpha\cup p_\beta, \dot{d}^\frown \langle\kappa, \dot{A}_\alpha\cap\dot{A}_\beta\rangle)\notin (j(\mathbb{P})\ast\dot{\mathbb{Q}}^{M}_{j(\mathcal{U})})/(\mathbb{P}\ast\dot{\mathbb{Q}}).$$
 But this contradicts our construction.
\end{proof}
\begin{lemma}
$b\in V[G_{\mathbb{R}}]$.
\end{lemma}
\begin{proof}
Note that $V[G_{\mathbb{R}}]\subseteq V[G_{j(\mathbb{R})}]$ as witnessed by the projection $\varrho: j(\MRB) \to \MRB.$
The follwoing also holds  in $V[G_{\MRB}]$, thanks to Lemma 5.2.
\[
V^{j(\MRB)/ G_{\MRB}} \subseteq V^{[(j(\mathbb{P})\ast\dot{\mathbb{Q}}^{M}_{j(\mathcal{U})})/(\mathbb{P}\ast\dot{\mathbb{Q}})]\ast\mathbb{U}^*}.
\]
It follows that $b \in V[G_{\MRB}][H\ast K],$ where $H\ast K$ is $[(j(\mathbb{P})\ast\dot{\mathbb{Q}}^{M}_{j(\mathcal{U})})/(\mathbb{P}\ast\dot{\mathbb{Q}})]\ast\mathbb{U}^*$-generic over $V[G_{\MRB}]$.
On the other hand:
\begin{itemize}
\item
$\mathbb{U}^*$ is $\kappa^{+}$-closed in $V[G_{\MRB}]$ (by Lemma 5.2).
\item
$(j(\mathbb{P})\ast\dot{\mathbb{Q}}^{M}_{j(\mathcal{U})})/G_{\mathbb{P}\ast\dot{\mathbb{Q}}}$ is $\kappa^+$-Knaster in $V[G_{\mathbb{R}}\ast K].$
\end{itemize}
It follows that neither $\mathbb{U}^*$, nor $(j(\mathbb{P})\ast\dot{\mathbb{Q}}^{M}_{j(\mathcal{U})})/G_{\mathbb{P}\ast\dot{\mathbb{Q}}}$ can add a branch through T, hence

 $b \in V[G_{\MRB}]$ as required.
\end{proof}
\section{Acknowledgements}
The authors would like to thank the anonymous referee for his/her careful reading and useful suggestions. The first author's research has been supported by a grant from IPM (No. 96030417). The second author's research has been partially supported by IPM.

School of Mathematics, Institute for Research in Fundamental Sciences (IPM), P.O. Box:
19395-5746, Tehran-Iran.

E-mail address: golshani.m@gmail.com
\\

Equipe de Logique Math\' ematique, Institut de Math\' ematiques de Jussieu, Universit\' e Paris Diderot, Paris, France.

E-mail address: rahmanmohammadpour@gmail.com

\end{document}